\documentclass[11pt]{amsart}

\usepackage{amsmath,amssymb,amsthm}
\usepackage{array}
\usepackage{graphicx}
\usepackage[hidelinks]{hyperref}

\graphicspath{{figures/}}

\numberwithin{equation}{section}
\numberwithin{figure}{section}

\theoremstyle{plain}
\newtheorem{theorem}{Theorem}[section]
\newtheorem{proposition}{Proposition}[section]
\newtheorem{lemma}{Lemma}[section]
\newtheorem{corollary}{Corollary}[section]
\newtheorem{conjecture}{Conjecture}[section]

\theoremstyle{definition}
\newtheorem{definition}{Definition}[section]

\theoremstyle{remark}
\newtheorem{remark}{Remark}[section]

\begin{document}

\title[Optimal (partial) transport to non-convex polygonal domains]
{Optimal (partial) transport to non-convex polygonal domains}

\author[S. Chen]{Shibing Chen}
\address{Shibing Chen, School of Mathematical Sciences,
  University of Science and Technology of China,
  Hefei, 230026, P.R. China.}
\email{chenshib@ustc.edu.cn}

\author[Y. Li]{Yuanyuan Li}
\address{Yuanyuan Li, Institute for Theoretical Sciences,
  Westlake University, Hangzhou, 310030, P.R. China.}
\email{lyyuan@westlake.edu.cn}

\author[J. Liu]{Jiakun Liu}
\address[Jiakun Liu]{School of Mathematics and Statistics,
  The University of Sydney,
  Camperdown, NSW 2006, Australia.}
\email{jiakun.liu@sydney.edu.au}

\thanks{The research of Chen was supported by the National Key R\&D Program of China
  (2022YFA1005400), the National Science Fund for Distinguished Young Scholars
  (No. 12225111), and NSFC No. 12141105. The research of Li was supported by
  China Postdoctoral Science Foundation (No. 2025M783144). The research of Liu
  was supported by ARC DP230100499 and FT220100368.}

\date{\today}

\begin{abstract}
  In this paper, we investigate optimal (partial) transport problems for which the
  target is a non-convex polygonal domain in \(\mathbb{R}^2\). For the complete
  optimal transport problem, we prove that the singular set is locally a smooth
  one-dimensional curve away from finitely many points. For the optimal partial
  transport problem, we prove that the free boundary is smooth away from finitely
  many singular points. In higher dimensions, we formulate two conjectures
  concerning the structure of singularities when the target is a non-convex
  polytope.
\end{abstract}

\maketitle
\setlength{\baselineskip}{16.4pt}
\setlength{\parskip}{3pt}

\section{Introduction}
Let $\Omega$ and $\Omega^*$ be two bounded domains in $\mathbb{R}^2$.
Let $m$ be a positive constant satisfying
\begin{equation*}
  m \leq \min\Big\{|\Omega|, \, |\Omega^*|\Big\},
\end{equation*}
where $|\cdot|$ denotes the two-dimensional Lebesgue measure.
A non-negative, finite Borel measure $\gamma$ on $\mathbb{R}^2 \times \mathbb{R}^2$ is a \emph{transport plan} (with mass $m$) from $\Omega$ to $\Omega^*$ if $\gamma(\mathbb{R}^2\times\mathbb{R}^2)  = m$ and, for every Borel set $E\subset\mathbb{R}^2$, the following conditions hold:
\begin{equation*}
  \begin{split}
    \gamma(E\times \mathbb{R}^2)  & \leq |E\cap\Omega|,   \\
    \gamma(\mathbb{R}^2 \times E) & \leq |E\cap\Omega^*|.
  \end{split}
\end{equation*}
A transport plan $\gamma$ is \textit{optimal} if it minimizes the cost functional
\begin{equation} \label{eq:mini}
  \int_{\mathbb{R}^2 \times \mathbb{R}^2} |x-y|^2\, d\gamma(x,y)
\end{equation}
among all transport plans, see \cite{V1,V2}.

When $|\Omega|=|\Omega^*|=m$, this is the complete optimal transport problem.
By Brenier's theorem \cite{Br1}, there exists a globally Lipschitz convex function $u:\mathbb{R}^2\rightarrow \mathbb{R}$
such that the optimal transport plan $\gamma=(Id\times Du)_\sharp \chi_{_\Omega}$, with
\begin{equation}\label{eq:bre1}
  (Du)_\sharp \chi_{_\Omega}= \chi_{_{\Omega^*}}
\end{equation}
and $Du(x) \in \overline{\Omega^*}$ for almost every $x \in \mathbb{R}^2$. The regularity of the potential function $u$ has been extensively studied, see \cite{C92,C92b,C96,CL2,CLW1,D91,CR,SY,U1}.
If the target domain $\Omega^*$ is non-convex, singularities may arise, see \cite{C92} for instance.
Let
\begin{equation}\label{eq:defA}
  A:=\{x\in \Omega: u \text{ is not differentiable at } x\}
\end{equation}
be the singular set.
The structure of the singular set in optimal transport was first investigated by Yu in \cite{Yu}, where he proved that if \(\Omega\) is convex and \(\Omega^*\) is a planar \(C^1\) domain satisfying certain mild conditions, then the singular set is a one-dimensional \(C^1\) manifold (a disjoint union of countably many \(C^1\) curves) except for a countable set. Subsequently, Figalli \cite{AFi2} established the same result for general non-convex planar domains.

Our first result gives a more precise description of the structure of the singular set $A$ when the target is a non-convex polygonal domain.

\begin{theorem}\label{thm:t111}
  Suppose \(\Omega\) is a bounded domain and \(\Omega^*\) is a non-convex polygonal domain in \(\mathbb{R}^2\) with $|\Omega|=|\Omega^*|$. Suppose \( u \) is a convex solution to \eqref{eq:bre1}. Then, the singular set \( A \) is either a finite set or, away from finitely many points
  \(\{x_i\}_{i=1}^n\), \(A\) is locally a smooth one-dimensional curve. Moreover, for each \(i = 1, \dots, n\), either \(x_i\) is an isolated singularity, or \(A\) is a union of finitely many Lipschitz curves near \(x_i\).
\end{theorem}

Optimal transport between polygonal domains has important practical applications. For example, in numerical computations, particularly in finite element methods, the generation of high-quality meshes on polygonal domains is crucial. Optimal transport can be used to design mesh generation algorithms that distribute mass over polygonal domains, improving the accuracy and efficiency of simulations. We refer the reader to \cite{Levy,Mer} and references therein.

\begin{remark}
  Let \(\{x_i\}_{i=1}^n\) be as in the above theorem. Denote \(F := \{x_i : i = 1, \dots, n\}\).
  If \(x \in A \setminus F\), then \(B_{r_x}(x) \cap A\) is a smooth curve for some small \(r_x > 0\).
  Because \(F\) is finite, for each \(x_i\), if it is not an isolated singular point, there exists a small \(r_i > 0\) such that
  \(B_{r_i}(x_i) \cap A = \cup_{j=1}^{k_i}\gamma_{ij}\) for some positive integer \(k_i\), where \(\gamma_{ij}\) is a Lipschitz curve that is smooth except at the point \(x_i\).
\end{remark}

When $m<\min\Big\{|\Omega|, \, |\Omega^*|\Big\}$, the minimization problem \eqref{eq:mini} is called the \emph{optimal partial transport problem}.
The word ``partial'' indicates that not all of the mass in $\Omega$ is transported to $\Omega^*$. The existence and uniqueness of the optimal transport plan have been established in \cite{CM}, under the assumption that $\Omega$ and $\Omega^*$ are separated by a straight line. For more general cases, the reader is referred to \cite{AF09,AFi,I} and references therein.
Let $U\subset \Omega$ be the subdomain in which the mass $m=|U|$ is transported to $V\subset\Omega^*$ by the optimal transport plan. The sets $\mathcal F:=\partial U\cap \Omega$ and $\mathcal F^*:=\partial V\cap\Omega^*$ are called the ``free boundaries'' of the problem. Our second main result of this paper concerns the regularity of the free boundary.

\begin{theorem}\label{thm:t222}
  Suppose \(\Omega\) is a bounded domain and \(\Omega^*\) is a non-convex polygonal domain in \(\mathbb{R}^2\),  separated from \(\Omega\) by a straight line. Let $m<\min\Big\{|\Omega|, \, |\Omega^*|\Big\}$ be the mass transported. Then the free boundary $\mathcal{F}$ is smooth away from finitely many singular points.
\end{theorem}

\begin{remark}
  The results of Theorems \ref{thm:t111} and \ref{thm:t222} also hold for general positive smooth densities. If the densities are only
  bounded above and below by positive constants, then, except for finitely many points, the singular set in Theorem \ref{thm:t111} and the free boundary in Theorem \ref{thm:t222} are locally $C^{1,\alpha}$ curves.
\end{remark}

The remainder of the paper is organized as follows. In Section \ref{sec:S2}, we introduce some tools and definitions that will be used in subsequent analysis. In Section \ref{sec:S3}, we investigate the structure of the singular set $A$ and prove Theorem \ref{thm:t111}. In Section \ref{sec:S4}, we study the regularity of the free boundary $\mathcal{F}$ and prove Theorem \ref{thm:t222}.
In Section \ref{sec:S5}, we discuss two conjectures concerning the structure of singularities in higher dimensions.
In Section \ref{sec:S6}, we sketch the proof of the \(C^{2,\alpha}\) estimate that completes the smoothness result.

\section{Preliminaries}\label{sec:S2}
Let \( X, Y \) be two bounded open sets in \( \mathbb{R}^2 \) satisfying \( |X| = |Y| \). By Brenier's theorem \cite{Br1}, there exists a convex function \( w: \mathbb{R}^2 \rightarrow \mathbb{R} \), globally Lipschitz, such that
\[ (Dw)_\sharp \chi_{_X} = \chi_{_Y}, \quad \text{ with } Dw(x) \in \overline{Y} \ \text{ for a.e. } x \in \mathbb{R}^2 .\]
Let $w^*$ denote the Legendre transform of $w,$ namely,
\[w^*(y):=\sup_{x\in\mathbb{R}^2}\{y\cdot x-w(x)\}, \quad y\in\mathbb{R}^2.\]

We need the following definition of sublevel sets. Suppose $w: \mathbb{R}^2\rightarrow \mathbb{R}$ is a convex function, whose graph contains no infinite straight line.

\begin{definition}
  The subgradient of $w$ at $y_0$ is defined by
  \begin{equation*}
    \partial^{-} w (y_0) = \{p\in\mathbb{R}^2 : w(y) \geq w(y_0) + (y-y_0)\cdot p \quad\forall\,y\in\mathbb{R}^2\},
  \end{equation*}
  and
  \begin{equation*}
    \partial^{-} w (E) = {\cup}_{y\in E} \partial^{-} w (y)  \quad \forall\, E\subset\mathbb{R}^2.
  \end{equation*}
\end{definition}

For any domain $D\subset \mathbb{R}^2$ and any non-negative function $f$ defined on $D,$
we say
\[\det\, D^2w=(\geq, \leq)\, f\quad\text{ in }D\]
in the Alexandrov sense, if
\[|\partial^-w(E)|=(\geq, \leq)\int_E f \quad \forall\, E\subset D.\]

\begin{definition}\label{def:defS}
  Given a point $y_0\in \mathbb{R}^2$ and a small constant $h>0$, we define the centered sublevel set of $w$ at $y_0$ with height $h$ as
  \begin{equation*}
    S^c_{h}[w](y_0) := \left\{y\in\mathbb{R}^2 : w(y) < w(y_0) + (y-y_0)\cdot \bar{p} + h\right\},
  \end{equation*}
  where $\bar{p}\in \mathbb{R}^2$ is chosen such that the center of mass of $S^c_{h}[w](y_0)$ is $y_0$.
\end{definition}

\begin{remark}
  If $w(0)=0$ and $w\geq 0$, by \cite[Remark 2.2]{CLW3}, we have that
  \begin{equation}\label{eq:c0est}
    w \leq Ch \quad \text{in}\ S^c_h[w](0),
  \end{equation}
  where $C>0$ depends only on the dimension $n.$
  By \cite[Lemma 2.1]{C96},
  $S^c_h[w](0)$ is balanced with respect to $0,$ namely,
  \begin{equation} \label{eq:balance1}
    y\in S^c_h[w](0)\quad \Longrightarrow \quad -Cy\in S^c_h[w](0)
  \end{equation}
  for a constant $C>0$ depending only on the dimension  $n.$
\end{remark}

\section{Singularities in optimal transportation}\label{sec:S3}

Suppose that \(\Omega, \Omega^*\) satisfy the hypotheses of Theorem \ref{thm:t111}.
Let the vertices of $\Omega^*$ be $\{b_i\in\mathbb{R}^2,\ i=1,\dots, m\}$, and let its edges be $[b_ib_{i+1}]$, with $b_{m+1}:=b_1$.
We use $[pq]$ to denote the closed segment joining $p$ and $q$, and we use $(pq)$ to denote the relative interior of $[pq]$. We also write $(pq]:=(pq)\cup \{q\}$. We call a vertex $b_i$ concave if  $B_r(b_i)\setminus\Omega^*$ is convex
for $r>0$ sufficiently small.

Let $u$ be the globally Lipschitz convex function satisfying \eqref{eq:bre1}. Similarly, by \cite{Br1} there exists
a globally Lipschitz convex function $v$ satisfying
\begin{equation*}
  (Dv)_\sharp \chi_{_{\Omega^*}}= \chi_{_{\Omega}},
\end{equation*}
with $Dv(y) \in \overline{\Omega}$ for almost every $y \in \mathbb{R}^2$.
Note that $v^*=u$ in $\Omega.$

By \eqref{eq:bre1}, the function \( u \) satisfies \(\det\, D^2 u \geq \chi_{_\Omega}\) in the Alexandrov sense. In dimension 2, it is well known that \( u \) is strictly convex in \(\Omega\), see, for instance, \cite{C93}, \cite[Theorem 2.5]{LW}, or \cite[Lemma 2.3]{Mooney}.
Note that for any \( x \in \Omega \), \( \partial^- u(x) \) is a bounded closed convex set.
Since \( v^* = u \) in \(\Omega\), we have
\begin{equation*}
  \partial^- u(x)=\{y\in\mathbb{R}^2: v(y)=x \cdot (y-p)+v(p)\} \quad \forall\, p\in \partial^- u(x),
\end{equation*}
namely, on the set $\partial^- u(x)$, $v$ agrees with an affine function that has gradient $x$.
By the strict convexity of $u$ in $\Omega$, we then have
\begin{equation}\label{eq:dif1}
  v \text{ is differentiable at any point } y \in \partial^{-} u(\Omega).
\end{equation}

Let
\[ \Omega' := \{x \in \Omega : \partial^- u(x) \cap \Omega^* \neq \emptyset\}. \]
We record the following observations.

\noindent $(i)$ \emph{The singular set $A\subset \Omega \setminus \Omega'$ and $|\Omega \setminus \Omega'|=0$.}

From the discussion following \cite[Proposition 2]{FK1}, \(\Omega'\) is an open set with \(|\Omega \setminus \Omega'| = 0\), and furthermore, \( u \) is strictly convex and smooth inside \(\Omega'\). Consequently, the singular set $A$ defined in \eqref{eq:defA} satisfies
\( A \subset \Omega \setminus \Omega'.\) \qed

\noindent $(ii)$ \( \forall\, x \in \Omega \setminus \Omega' \),
\begin{equation} \label{eq:obs2}
  \text{ext}(\partial^- u(x))\subset\partial\Omega^*.
\end{equation}
Indeed, given any \( x \in \Omega \setminus \Omega' \), we have \( \partial^- u(x) \cap \Omega^* = \emptyset \). By the discussion preceding \cite[Proposition 2]{FK1}, we have that \(\text{ext}(\partial^- u(x)) \subset \partial \Omega^*\).
\noindent $(iii)$ \emph{ A localisation property, \eqref{eq:loc1}.}

Let \( x_0 \in A \). Fix any  \( p \in  \text{ext}(\partial^- u(x_0)).\)
Suppose  \(p\in [b_{i_0} b_{i_0+1}] \) for some $i_0=1,\dots, m$.
By \eqref{eq:dif1}, \( v \) is differentiable at \( p \), and one can choose \( r_0 \) small enough such that
\begin{equation}\label{eq:r0small1}
  \partial^- v(B_{r_0}(p)\cap \Omega^*) \subset \Omega
\end{equation}
and that
\begin{equation}\label{eq:r0small}
  r_0 < \min\big\{\text{dist}(p, [b_i b_{i+1}]): 1 \leq i \leq m, \ p \notin [b_i b_{i+1}]\big\}.
\end{equation}

Arguing as in \cite[Lemma 3.3]{CL3}, we have the following localisation lemma.
\begin{lemma}\label{lem:loclem} Let $x_0, p, r_0$ be as above.
  Then, there exist positive constants \( h_0 \) and \( \delta_0\in (0, 1) \) such that
  \begin{equation}\label{eq:loc1}
    S_h^c[v](y) \subset B_{r_0}(p) \quad \forall\, y \in B_{\delta_0}(p) \cap \overline{\Omega^*},\ \ \forall\, h \leq h_0.
  \end{equation}
\end{lemma}
\begin{proof}
  By translating the coordinates and subtracting a constant, we may assume that \( x_0 = 0 \) and \( v \geq 0 \), with \( v(p) = 0 \). Then, \( \partial^- u(0) = \{ v = 0 \} \).

  Suppose, to the contrary, that there are no \( h_0 \) and \( \delta_0\in (0, 1) \) for which \eqref{eq:loc1} holds. Then there are sequences \( y_k \in B_{\frac{1}{k}}(p) \) and \( h_k \leq \frac{1}{k} \) such that \( S_{h_k}^c[v](y_k) \subset B_{r_0}(p) \) fails.

  By the contradiction assumption and \eqref{eq:balance1}, we have that \( I_k \subset S_{h_k}^c[v](y_k) \) for some segment \( I_k \) centered at \( y_k \), satisfying \( |I_k| \geq C_1 \) for some constant \( C_1 > 0 \) independent of \( k \). Note that \( h_k \to 0 \) and \( y_k \to p \) as \( k \to \infty \).

  Up to a subsequence, we may assume that \( I_k \) converges to \( I \), a segment centered at \( p \) with length at least \( C_1 \). By \eqref{eq:c0est}, we have that \( v \leq C_2 h_k \) in \( S_{h_k}^c[v](y_k) \) for some constant \( C_2 \) independent of \( k \). It follows that \( v \leq C_2 h_k \) on \( I_k \), and passing to the limit, we have that \( 0 \leq v \leq 0 \) on \( I \). Hence, \( I \subset \{ v = 0 \} \), contradicting the fact that \( p \) is an extreme point of \( \{ v = 0 \} \).
  Therefore, there must exist \(h_0\) and \( \delta_0\in (0, 1) \) such that \eqref{eq:loc1} holds.
\end{proof}

\begin{lemma} \label{lem:locloc}
  Let $x_0\in A.$ For any $q \in  \partial^- u(x_0)\cap \partial \Omega^*,$ there exists $0<r$ such that the function \( v \) satisfies
  \begin{equation*}
    \det\, D^2 v = \chi_{_{B_{r}(q) \cap \Omega^*}}\quad \text{ in } B_{r}(q),
  \end{equation*}
  in the Alexandrov sense.
\end{lemma}

\begin{proof}
  By \eqref{eq:dif1}, \( v \) is differentiable at \( q \), and one can choose \( r \) small enough such that  \( \partial^-v(y) \subset \Omega \) for any
  \( y \in B_{r}(q) \cap \Omega^* \). Hence, according to \cite[Theorem 3.3]{FK1},
  \( v \) is strictly convex and smooth in \( B_{r}(q) \cap \Omega^* \), which implies
  \begin{equation}\label{eq:ale1}
    \det\, D^2v(y) = 1\quad \forall\, y \in B_{r}(q) \cap \Omega^*.
  \end{equation}
  On the other hand, for \( y \in B_{r}(q) \setminus \Omega^* \), let \( x \in  \partial^-v(y)\subset\Omega \).
  It follows that \( y \in \partial^-u(x) \). If \( x \in \Omega' \), since \( u \) is strictly convex and smooth in \( \Omega' \),
  we have \( y=Du(x) \in \Omega^* \), contradicting \( y \in B_{r}(q) \setminus \Omega^* \). Therefore, \( x \) must be in \( \Omega \setminus \Omega' \).
  This implies that
  \begin{equation}\label{eq:ale2}
    |\partial^-v(B_{r}(q) \setminus \Omega^*)| = 0.
  \end{equation}
  By combining \eqref{eq:ale1} and \eqref{eq:ale2}, it follows that
  \( \det\, D^2 v = \chi_{_{B_{r}(q) \cap \Omega^*}} \) within \( B_{r}(q) \), in the Alexandrov sense.
\end{proof}

\begin{lemma}\label{lem:extonly1}
  For any \(x_0 \in A\), \(\partial^- u(x_0) \cap \partial \Omega^* = \text{ext}(\partial^- u(x_0)).\)
\end{lemma}

\begin{proof}
  Suppose, for the sake of contradiction, that there exists \(q \in \partial^- u(x_0) \cap \partial \Omega^*\) that is not an extreme point of \(\partial^- u(x_0)\). By \eqref{eq:dif1}, we have that \(v\) is differentiable at \(q\). Take a sequence \(r_k \rightarrow 0\). Since $Dv(q)=x_0\in\Omega$ and
  \(v\) is convex and differentiable at \(q\), we have that \(\partial^-v\left(B_{r_k}(q)\right) \subset \Omega\) for large \(k\) and that it converges to \(x_0\) as \(k \rightarrow \infty\) in Hausdorff distance. By Lemma \ref{lem:locloc}, we have that \(\partial^- v\left(B_{r_k}(q)\right)\) has positive Lebesgue measure. Hence, we can find an \(x_k \in \partial^- v\left(B_{r_k}(q)\right)\) such that \(u\) is differentiable at \(x_k\). Note that \(x_k \rightarrow x_0\) as \(k \rightarrow \infty\). On one hand, by the convexity of \(u\), passing to a subsequence, we find that \(Du(x_k) \) converges to an extreme point of \(\partial^- u(x_0)\). Denote \(q_0 := \lim_{k \rightarrow \infty} Du(x_k)\).
  On the other hand, since \(Du(x_k) \in B_{r_k}(q)\), we have that \(q_0 = q \notin \text{ext}(\partial^- u(x_0))\), which is a contradiction.
\end{proof}

It follows from the above lemma that
\begin{corollary}\label{cor:sconvex1}
  For any $x_0\in A$ and any $i\in\{1, \dots, m\}$,
  $\text{ext}(\partial^- u(x_0)) \cap [b_ib_{i+1}]$ is either empty or a singleton.
\end{corollary}
\begin{proof}
  Suppose to the contrary that, for some $i\in\{1, \dots, m\}$, $\text{ext}(\partial^- u(x_0)) \cap [b_ib_{i+1}]$
  contains at least two distinct points, denoted by \(q_1\) and \(q_2\). Then the open segment \((q_1q_2) \subset \partial^- u(x_0)\cap \partial\Omega^*\), which contradicts Lemma \ref{lem:extonly1}.
\end{proof}

\begin{remark}
  It follows from Corollary \ref{cor:sconvex1} and \eqref{eq:obs2} that for any \( x \in A \), \( \partial^- u(x) \) must be a closed convex polygon, which may also be a one-dimensional segment.
\end{remark}

We now divide the singular set $A$ into two parts based on whether the set $\partial^- u(x)$ contains a vertex of $\Omega^*$.
Let
\[ \Sigma_1 := \{x\in A : b_i\in \partial^- u(x) \text{ for some } i=1,\dots,m\},\quad \text{ and }\ \Sigma_2 :=A\setminus \Sigma_1. \]

For any \( x \in \Sigma_2 \), Corollary \ref{cor:sconvex1} and \eqref{eq:obs2} imply that the extreme points of \(\partial^- u(x)\) lie in the relative interiors of the edges of \(\Omega^*\).

\begin{lemma}\label{lem:finite}
  $\Sigma_1$ consists of finitely many points.
\end{lemma}

\begin{proof}
  Given any \( x \in \Sigma_1 \), by \eqref{eq:dif1}, \( v \) is differentiable at any \( y \in \partial^- u(x) \). Hence, \( x = Dv(b_i) \) for some \( i=1,\dots,m \). This implies that \( \Sigma_1 \) is a finite set with cardinality at most $m$.
\end{proof}

Next, we decompose \(\Sigma_2\) into \(\Sigma_2 = \Sigma'_2 \cup \Sigma''_2\) such that
\[ \Sigma'_2 := \{ x\in \Sigma_2 : \partial^- u(x) \text{ touches exactly two open edges of } \Omega^*\}, \text{ and } \Sigma''_2 = \Sigma_2 \setminus \Sigma'_2. \]

\begin{lemma}\label{lem:fini222}
  \(\Sigma''_2\) consists of finitely many points.
\end{lemma}

\begin{proof}
  Let \(x \in \Sigma''_2\). By the definition of \(\Sigma''_2\) and Corollary \ref{cor:sconvex1}, \(\partial^-u(x)\) must touch at least three open edges of \(\Omega^*\), say, at \(p_x^1 \in (b_{i_1}b_{i_1+1})\), \(p_x^2 \in (b_{i_2}b_{i_2+1})\), and \(p_x^3 \in (b_{i_3}b_{i_3+1})\), respectively. Then \(v = \ell\) for some affine function \(\ell\) in the open triangle \(\triangle_x\) with vertices \(p_x^k\), \(k=1,2,3\). Note that \(D\ell = x\).

  Let \(z\) be another point in \(\Sigma''_2\), and suppose \(\partial^-u(z)\) touches the same three open edges of \(\Omega^*\), at \(\bar{p}_z^1 \in (b_{i_1}b_{i_1+1})\), \(\bar{p}_z^2 \in (b_{i_2}b_{i_2+1})\), and \(\bar{p}_z^3 \in (b_{i_3}b_{i_3+1})\), respectively. Then \(v = \ell'\) for some affine function \(\ell'\) in the open triangle \(\triangle'_z\) with vertices \(\bar{p}_z^k\), \(k=1,2,3\). Note that \(D\ell' = z \neq x\).

  The planar geometry forces \(\triangle'_z\) to intersect \(\triangle_x\), contradicting the differentiability of \(v\) in \(\triangle'_z \cap \triangle_x\); see \eqref{eq:dif1}. Hence \(x = z\).
  The above discussion implies that \(\Sigma''_2\) is a finite set with cardinality at most \(\frac{m(m-1)(m-2)}{6}\).
\end{proof}

\begin{figure}[htbp]
  \centering
  \includegraphics[width=4in,keepaspectratio]{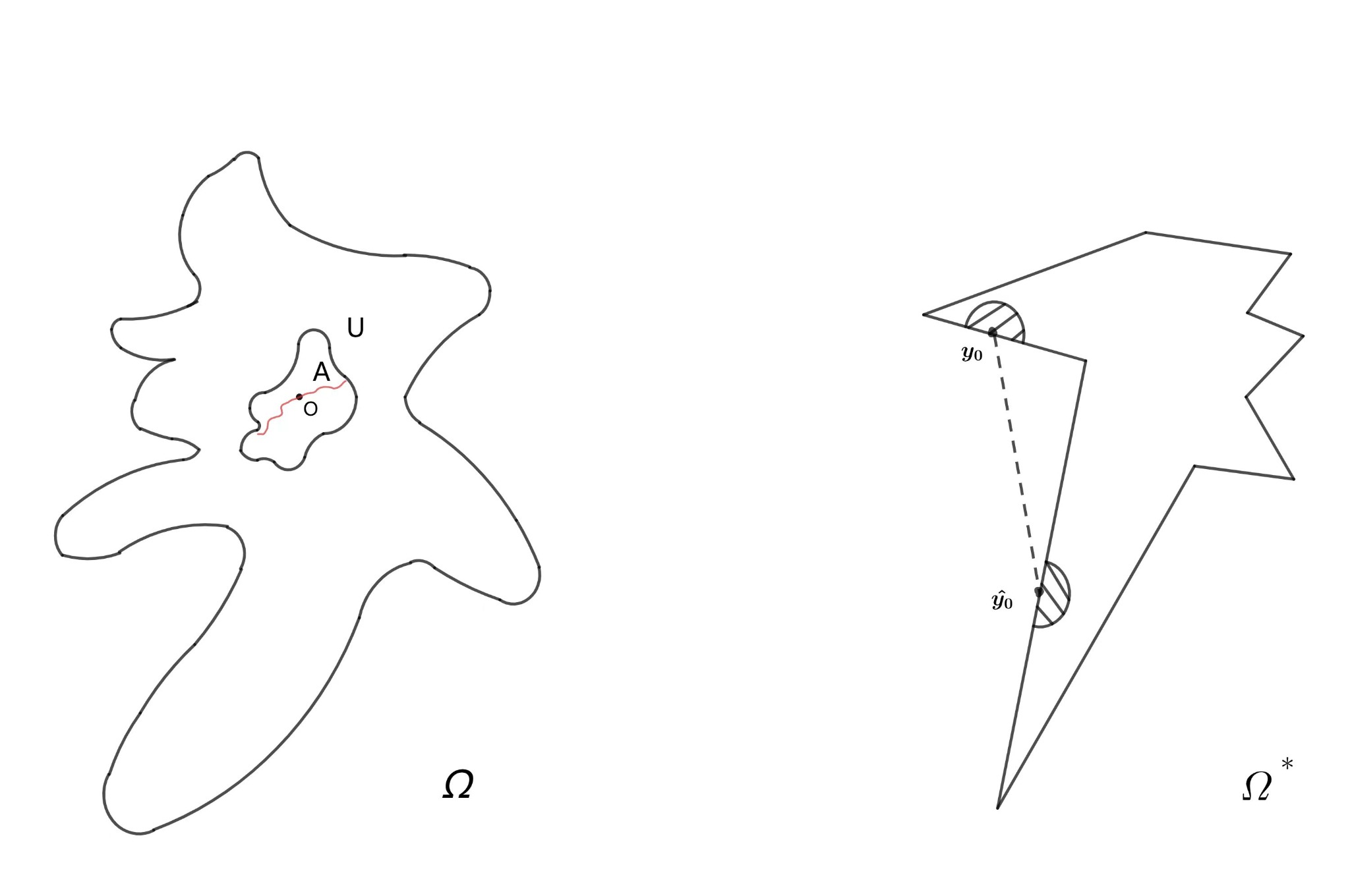}
  \caption{A local configuration for the singular set near two target edges.}
  \label{fig:singular-segment}
\end{figure}

\begin{proposition}\label{prop:propmain}
  For any \(x \in \Sigma'_2\), there exists a small constant \(r_x\) such that \(B_{r_x}(x) \cap A\) is a smooth curve.
\end{proposition}

\begin{proof}
  Let \(x_0 \in \Sigma'_2\). After translating coordinates, we may assume \(x_0 = 0\). By the definition of \(\Sigma'_2\), \(\partial^-u(0)\) is a segment with endpoints on two open edges of \(\Omega^*\), denoted by \((b_{i_1}b_{i_1+1})\) and \((b_{i_2}b_{i_2+1})\). Let \(y_0 := \partial^-u(0) \cap (b_{i_1}b_{i_1+1})\) and \(\hat{y}_0 := \partial^-u(0) \cap (b_{i_2}b_{i_2+1})\); see Figure~\ref{fig:singular-segment}.

  Denote \(V_1 := B_r(y_0) \cap \overline{\Omega^*}\) and \(V_2 := B_r(\hat{y}_0) \cap \overline{\Omega^*}\).
  By taking \(r>0\) sufficiently small, we have that \(V_1\) and \(V_2\) are convex and \(\text{dist}(V_1, V_2) > 0\).
  Denote \(U_i := \partial^- v(V_i)\) for \(i = 1, 2\).

  By Lemma \ref{lem:loclem}, there exist constants \( h_0 > 0 \) and \( \delta_0 > 0 \) such that
  \begin{equation}\label{eq:locc11}
    S_h^c[v](y_0) \subset B_r(y_0) \quad \forall\, y \in B_{\delta_0}(y_0) \cap \overline{V_1},\ \ \forall\, h \leq h_0.
  \end{equation}
  Then, by Lemma \ref{lem:locloc}, we have
  \begin{equation}\label{eq:alx11}
    \det D^2 v = \chi_{_{V_1}} \quad \text{in } B_r(y_0).
  \end{equation}
  Now, by \eqref{eq:locc11} and \eqref{eq:alx11}, we can invoke the proof of \cite[Theorem 7.13]{CM} to obtain quantitative uniform convexity of \( v \) within \( B_{\delta_0}(y_0) \cap \overline{V_1} \). This implies that for any two points \( y, \tilde{y} \in B_{\delta_0}(y_0) \cap \overline{V_1} \), the following inequality holds:
  \begin{equation}\label{eq:quantc1}
    |Dv(y) - Dv(\tilde{y})| \geq C |y - \tilde{y}|^a
  \end{equation}
  for some constants \( a > 2 \) and \( C > 0 \).
  Similarly, $v$ is also strictly convex in $B_{\delta_0}(\hat y_0) \cap \overline{V_2}$ for $\delta_0$ small.
  It follows that
  \[ B_{r_1}(0) \subset U := U_1 \cup U_2 \]
  for some sufficiently small \(r_1 > 0\).
  Denote \(V = V_1 \cup V_2\). We have that \((Du)_{\sharp}\chi_{_U} = \chi_{_V}\).

  To prove the \(C^\infty\) regularity of \(A\), we follow the methods developed in \cite{CL3,CLW3}.

  Let \(u_i(x) = \sup_{y \in V_i}\{y \cdot x - v(y)\}\) for \(x \in \mathbb{R}^2\), \(i = 1, 2\).
  Then \(u_i = u\) in \(U_i\), \(i = 1, 2\).
  The singular set of \(u\) in \(B_{r_1}(0)\) is characterized
  by
  \begin{equation}\label{eq:levdef}
    A \cap B_{r_1}(0) = \{u_1 = u_2\} \cap B_{r_1}(0) = \partial U_i \cap B_{r_1}(0), \ i = 1, 2.
  \end{equation}

  Since \( u_1^* = v \) in \( V_1 \), by \eqref{eq:quantc1}
  we can apply the result of \cite[Remark 7.10]{CM} to deduce that \( u_1 \in C^{1,\beta}(B_{r_2}(0) \cap \overline{U_1}) \) for some \( \beta \in (0, 1) \) and \( r_2 > 0 \). Similarly, we can show that \( u_2 \in C^{1,\beta}(B_{r_3}(0) \cap \overline{U_2}) \) for some \( \beta \in (0, 1) \) and \( r_3 > 0 \). It follows from \eqref{eq:levdef} and the implicit function theorem that \( A \) is \( C^{1,\beta} \) near \( x_0 \).

  From \eqref{eq:levdef} it also follows that the unit inner normal of \(U_2\) (which is also the unit normal of \(A\)) at \(x \in A \cap B_{r_1}(0)\) is given by
  \begin{equation}\label{eq:normf}
    \nu_{_A}(x) = \frac{Du_2(x) - Du_1(x)}{|Du_2(x) - Du_1(x)|}.
  \end{equation}
  Denote $y = Du_1(x)$, $\hat y = Du_2(x)$. By Corollary \ref{cor:sconvex1}, the segment \( y\hat y \) touches
  \((b_{i_1}b_{i_1+1})\) and \((b_{i_2}b_{i_2+1})\) only at \( y \) and \( \hat y \), respectively. It follows that
  \(\nu_{_A}(x) = \frac{\hat y - y}{|\hat y - y|}\) satisfies
  \begin{equation}\label{eq:oblique1}
    -\nu_{_A}(x) \cdot \nu_{_{V_1}}(y) > 0,\qquad  \nu_{_A}(x) \cdot \nu_{_{V_2}}(\hat y) > 0,
  \end{equation}
  where \(\nu_{_{V_1}}(y)\) and \(\nu_{_{V_2}}(\hat y)\) are the unit inner normals of \(V_1\) and \(V_2\) at \(y\) and \(\hat y\), respectively.

  Note that \( u_i\), \(i=1, 2 \), satisfy the following conditions:
  \begin{equation}\label{eq:uieq}
    \left\{
    \begin{array}{rl}
      \det\, D^2 u_i \!\! & = \ \chi_{_{U_i}} \quad \text{in } \mathbb{R}^2, \\
      D u_i(U_i) \!\!     & = \ V_i.
    \end{array}
    \right.
  \end{equation}
  Moreover, \( v \) satisfies
  \begin{equation}\label{eq:veq11}
    \det\, D^2 v = \chi_{_{B_{r}(y_0) \cap V_1}} \ \text{in}\ B_{r}(y_0)
  \end{equation}
  and
  \begin{equation*}
    \det\, D^2 v = \chi_{_{B_{r}(\hat y_0) \cap V_2}} \ \text{in}\ B_{r}(\hat y_0).
  \end{equation*}

  By \eqref{eq:oblique1}, up to an affine transformation (see \cite[(3.2)]{CLW3}), we may assume \(\nu_{_{V_1}}(y) = -\nu_{_A}(x) = e_2\).
  By \eqref{eq:uieq} and \eqref{eq:veq11}, we can apply Proposition \ref{prop:aprop1} $(i)$ to obtain that \( u_1 \) is \( C^{1, 1-\epsilon} \) along \( \partial U_1 \) near 0, for any $\epsilon>0$. Similarly, \( u_2 \) is \( C^{1, 1-\epsilon} \) along \( \partial U_2 \) near 0.
  Then, by \eqref{eq:normf}, we have that \( A \), and hence \( \partial U_i \) for \( i = 1, 2 \), is \( C^{1, 1-\epsilon} \) near 0. Finally, by applying
  Proposition \ref{prop:aprop1} $(ii)$, we conclude that each \( u_i \), \(i=1,2\), is \( C^{2,\alpha} \) in \( B_{\bar r}(0) \cap U_i \) for some small $\bar r>0$ and $\alpha\in(0,1)$. By a bootstrap argument and the Schauder estimate for uniformly elliptic equations \cite{GT}, we further conclude that each \( u_i \), \(i=1,2\), is smooth in \( B_{\bar r}(0) \cap U_i \). Hence, by \eqref{eq:normf}, \( A \) is smooth near 0.
\end{proof}

\begin{figure}[htbp]
  \centering
  \includegraphics[width=4in,keepaspectratio]{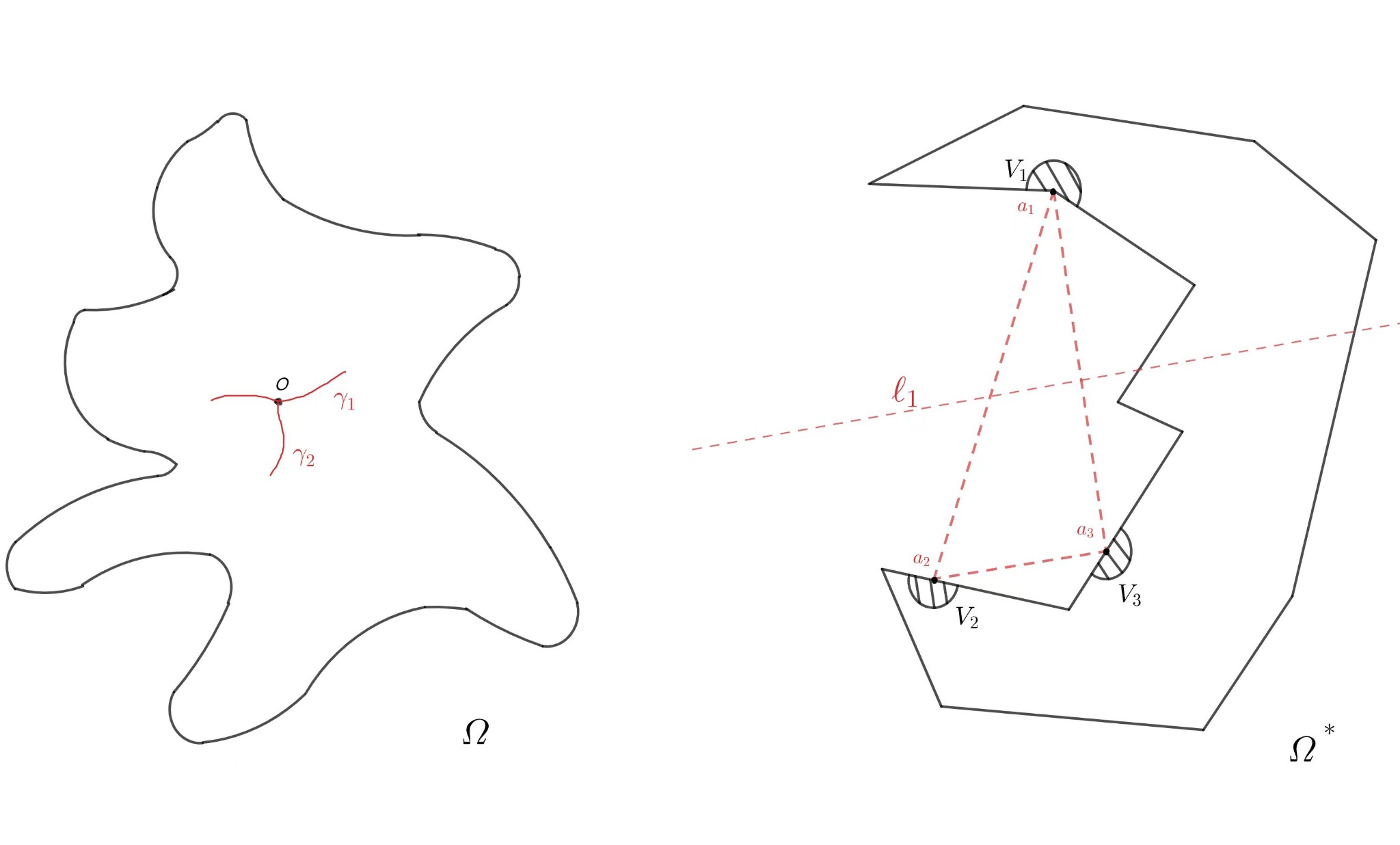}
  \caption{A local separation of one target component from the others.}
  \label{fig:separation}
\end{figure}

\begin{proposition}\label{prop:propmain1}
  For any \(x \in \Sigma_1 \cup \Sigma''_2\), there exists a small constant \(r_x\) such that either \(B_{r_x}(x) \cap A = \{x\}\), or
  \(B_{r_x}(x) \cap A\) is a union of finitely many Lipschitz curves.
\end{proposition}

\begin{proof}
  Let \(x_0 \in \Sigma_1 \cup \Sigma''_2\).
  We may assume that \(B_{\delta}(x_0) \cap A \setminus \{x_0\}\ne \emptyset\) for any $\delta>0,$ as otherwise $x_0$ would be an isolated singularity, and the proof is done.
  By Lemma \ref{lem:extonly1} and Corollary \ref{cor:sconvex1},  we have that \(\partial^-u(x_0) \cap \partial \Omega^*\)
  consists of finitely many points \(a_i\), \(i = 1, \dots, l\).

  After translating coordinates, we may assume \(x_0 = 0\). By subtracting a constant, we may also assume \(v = 0\) on \(\partial^-u(0)\) and \(v \geq 0\) on \(\mathbb{R}^2\).
  For $r_1>0$, define
  \[V_i := B_{r_1}(a_i) \cap \overline{\Omega^*},\quad i = 1, \dots, l,\qquad
    V := \cup_{i=1}^l V_i. \]
  Take \(r_1\) small enough such that \(V_i\), \(i = 1, \dots, l\), are disjoint.
  From \eqref{eq:dif1}, \(v\) is differentiable at any \(y \in \partial^-u(0)\).
  By the continuity of the subdifferential of a convex function,
  we have \(\partial^-v(V_i) \subset \Omega\), \(i = 1, \dots, l\), provided \(r_1\) is sufficiently small.
  Let \(U := \cup_{i=1}^l \partial^-v(V_i)\).
  Then \((Du)_\sharp \chi_{_U} = \chi_{_V}\).

  Since \(v(y) = 0\) only if \(y \in \partial^-u(0)\), by convexity, there exists a constant \(\eta > 0\) such that
  for any \(y \in \Omega^* \setminus V\), if \(z \in \partial^-v(y)\) then \(|z| > \eta\).
  It follows that there exists a small positive \(r < \eta\) such that \(B_r(0) \subset U\).
  Note that there exists a straight line $\ell_i$ separating $V_i$ from $\cup_{j\ne i} V_j$,
  provided \(r_1\) is sufficiently small, see Figure~\ref{fig:separation}.
  By \cite[Proposition 5.1]{KM1} or \cite[Theorem 5.1]{CL1}, we have that \(B_r(0)\subset\Omega\) is divided into two open sets \(B'\) and \(B''\) by a Lipschitz curve \(\gamma_1\), satisfying that \(Du(B') \subset V_1\)
  for a.e. \(x \in B'\), and that \(Du(B'') \subset \cup_{i=2}^l V_i\) for a.e. \(x \in B''\). Note that \(\gamma_1 \subset A\).

  If \(a_1\) is a vertex of \(\partial \Omega^*\) where \(\Omega^*\) is locally concave, then
  \begin{equation}\label{eq:exdivide}
    \text{the segment } [a_1 a_2] \ \text{divides} \ B_{r_1}(a_1) \setminus \Omega^* \ \text{into two disjoint parts.}
  \end{equation}
  Suppose that there exists \(\tilde{x} \in A \cap B'\), then \(\tilde{x} \neq 0\) and
  \begin{equation}\label{eq:divide1}
    \text{the convex set } \partial^-u(\tilde{x}) \ \text{contains at least two extreme points.}
  \end{equation}
  Since \(Du(B') \subset V_1\) for a.e. \(x \in B'\), we have \(\partial^-u(\tilde{x}) \subset B_{r_1}(a_1)\). Since the extreme points of \(\partial^-u(\tilde{x})\) are on \(\partial \Omega^*\), by Lemma \ref{lem:extonly1}, Corollary
  \ref{cor:sconvex1}, \eqref{eq:exdivide}, and \eqref{eq:divide1}, we have that \(\partial^-u(\tilde{x}) \cap [a_1 a_2] \neq \emptyset\). It follows that \(\partial^-u(\tilde{x}) \cap \partial^-u(0) \neq \emptyset\). Taking any \(z \in \partial^-u(\tilde{x}) \cap \partial^-u(0)\),
  we have \(\tilde{x} = Dv(z) = 0\), which contradicts the assumption \(\tilde{x} \neq 0\). Hence, in this case, \(A \cap B' = \emptyset\).

  Otherwise, \(V_1 = B_{r_1}(a_1) \cap \Omega^*\) is a convex set, provided \(r_1\) is small enough.
  Since \(Du(B') \subset V_1\) for a.e. \(x \in B'\), we have that \(B' \subset U_1 := \partial^-v(V_1)\).
  Since \((Du)_\sharp \chi_{_{U_1}} = \chi_{_{V_1}}\), and the target domain \(V_1 \subset \mathbb{R}^2\) is convex,
  by Caffarelli's regularity theory \cite{C92}, we have that \(u\) is \(C^1\) in \(U_1\). Hence,
  \(u\) is \(C^1\) in \(B'\), which also implies that \(A \cap B' = \emptyset\).

  Repeating the above process, we can conclude that \(A \cap B_r(0)\) is a union of finitely many Lipschitz curves.
\end{proof}

Combining the above results, we can now prove Theorem \ref{thm:t111}.

\begin{proof}[Proof of Theorem \ref{thm:t111}]
  Since \(A = \Sigma_1 \cup \Sigma'_2 \cup \Sigma''_2\), Lemma \ref{lem:finite}, Lemma \ref{lem:fini222}, and Proposition \ref{prop:propmain} imply that the singular set \( A \) is either a finite set or, away from finitely many points, is locally a smooth one-dimensional curve. By Proposition \ref{prop:propmain1}, we have that near these exceptional points, \( A \) is a union of finitely many Lipschitz curves.
\end{proof}

\begin{remark}
  The argument developed in this paper can also be used to study the problem in general dimensions when the target is the union of two bounded convex domains, which may overlap. Indeed, by combining the methods developed in \cite{CL3}, we obtain the following result. Let $\Omega\subset\mathbb{R}^n$ be a bounded domain, and let $\Omega^* = \Omega^*_1 \cup \Omega^*_2$ be the union of bounded convex domains $\Omega^*_1$ and $\Omega^*_2$ in $\mathbb{R}^n$. Let $u$ be a convex function solving $(Du)_\sharp \chi_{_\Omega}= \chi_{_{\Omega^*}}$, and let $A$ be the singular set of $u$ in $\Omega$. Then, for any $x\in A$, there exists a small constant $r_x>0$ such that $A\cap B_{r_x}(x)$ is an $(n-1)$-dimensional $C^{1,\alpha}$ submanifold of $\mathbb{R}^n$. Furthermore, if $\Omega_i^*$, $i=1, 2$, are uniformly convex and smooth, then $A\cap B_{r_x}(x)$ is smooth as well.
\end{remark}

\section{Free boundaries in optimal partial transport}\label{sec:S4}
Suppose \(\Omega\) is a bounded domain and \(\Omega^*\) is a non-convex polygonal domain in \(\mathbb{R}^2\), separated by a straight line \(L\). By a change of coordinates, we may assume that
\[ L = \{x = (x_1, x_2) \in \mathbb{R}^2 : x_1 = 0\}, \]
\(\Omega \subset \{x_1 < 0\}\), and \(\Omega^* \subset \{x_1 > 0\}\).
Suppose that the vertices of \(\Omega^*\) are given by \(b_i \in \mathbb{R}^2\), \(i = 1, \dots, m\), and the edges are given by \([b_i b_{i+1}]\), \(i = 1, \dots, m\), with \(b_{m+1} := b_1\).

For a fixed constant \(m\) satisfying \(m < \min\Big\{|\Omega|, \, |\Omega^*|\Big\}\), it is shown in \cite{CM} that the optimal transport plan \(\gamma\), namely, the minimizer of \eqref{eq:mini}, is characterized by:
\begin{equation*}
  \gamma = (\text{Id} \times T)_{\#}\chi_{_U}  = (T^{-1} \times \text{Id})_{\#}\chi_{_V},
\end{equation*}
where \(T\) is the optimal transport map from the active domain \(U \subset \Omega\) to the active target \(V \subset \Omega^*\).
Moreover, there exists a globally Lipschitz convex function \(v: \mathbb{R}^2 \rightarrow \mathbb{R}\) such that \(T^{-1} = Dv\), and
\begin{equation*}
  (Dv)_{\#}\left(\chi_{_V} + \chi_{_{\Omega \setminus U}}\right) = \chi_{_{\Omega}}
\end{equation*}
with \(Dv(y) \in \Omega\) for almost every \(y \in \mathbb{R}^2\), and
\begin{equation*}
  v(x) = \frac{1}{2}|x|^2 + C\quad \text{for } x \in \Omega \setminus U,
\end{equation*}
for some constant \(C\).

Denote \( \bar{u} = v^* \), the Legendre transform of \( v \). It is well known that \( (D\bar{u})_{\#}\chi_{_U} = \chi_{_V} \). It is also straightforward to check that
\begin{equation}\label{eq:main1}
  \det\, D^2\bar{u} \geq 1 \quad \text{in}\ \Omega
\end{equation}
in the sense of Alexandrov. Furthermore, it is well known that in dimension \(2\), if a convex function satisfies \eqref{eq:main1}, then it must be strictly convex in \( \Omega \).

It is also proved in \cite{CM} that the free boundary \(\mathcal{F} := \partial U \cap \Omega\) is the graph of a semiconvex function (hence a Lipschitz function) defined over the line \(L\).
Moreover, the interior ball property holds \cite[Corollary 2.4]{CM}, namely, for any \(x \in U\), if \(\bar{u}\) is differentiable at \(x\), then
\begin{equation}\label{eq:inball}
  B_{|x-y|}(x) \cap \Omega^* \subset V, \quad\text{ and }\quad B_{|x-y|}(y) \cap \Omega \subset U
\end{equation}
where \(y = D\bar{u}(x)\).

For any \(x \in \mathcal{F}\), note that \(\partial^-\bar{u}(x)\) is a closed convex set.
After translating coordinates, we may assume \(x = 0 \in \mathcal{F}\).
By adding a constant, we may also assume that \(v \geq 0\) on \(\mathbb{R}^2\), \(v = 0\) on
\(\partial^-\bar{u}(0)\), and
\begin{equation}\label{eq:qua1}
  \bar{u} = v = \frac{1}{2}|x|^2 \quad\text{ in } \ \ \Omega \setminus U.
\end{equation}
Since \(\bar{u}\) is strictly convex in \(\Omega\), and \(\bar{u} = v^*\), it follows that
\begin{equation} \label{eq:c1lemma1}
  v\ \text{is differentiable at any}\ y \in \partial^-\bar{u}(0).
\end{equation}

\begin{lemma}\label{lem:pos1}
  \(\partial^-\bar{u}(0)\) is a bounded convex polygonal set satisfying \(\partial^-\bar{u}(0) \cap \big((\Omega \setminus \overline{U}) \cup V\big) = \emptyset\) and
  \(\text{ext}(\partial^-\bar{u}(0)) \subset \{0\} \cup \partial V\).
\end{lemma}

\begin{proof}
  If \(\partial^-\bar{u}(0) \cap \big((\Omega \setminus \overline{U}) \cup V\big) \neq \emptyset\), by
  \cite[Theorem 3.3]{FK1}, we have that \(\bar{u}\) is differentiable at $0$, contradicting the fact that \(0 \in \mathcal{F}\), on which \(\bar{u}\) cannot be differentiable. Hence, we obtain \(\partial^-\bar{u}(0) \cap \big((\Omega \setminus \overline{U}) \cup V\big) = \emptyset\).

  Next, we show that \(\partial^-\bar{u}(0)\) is bounded. Suppose not, since $0\in\partial^-\bar{u}(0)$, by convexity of \(\partial^-\bar{u}(0)\), there exists $e \in \mathbb{S}^{1}$ such that the half-line \(\{te : t \geq 0\} \subset \partial^-\bar{u}(0),\) which implies
  \[ 0\in \partial^-v(te) \quad \forall\,t\geq 0. \]
  By the monotonicity of $\partial^-v$, for any \(y \in \mathbb{R}^2\) and \(x \in \partial^{-}v(y)\), we then have that
  \[ (x - 0) \cdot (y - te) \geq 0 \quad \forall\,t\geq 0.  \]
  Letting \(t \to\infty \), we obtain \(x \cdot e \leq 0\).
  Given the arbitrariness of \(y\), this implies \(\Omega \subset \partial^{-} v(\mathbb{R}^2) \subset \{x \cdot e \leq 0\}\), contradicting the fact that \(0 \in \Omega\). Therefore, \(\partial^-\bar{u}(0)\) must be bounded.

  Finally, by \eqref{eq:qua1} and the discussion above \cite[Proposition 2]{FK1}, we have that
  \(\text{ext}(\partial^-\bar{u}(0)) \subset \{0\} \cup \partial V\).
\end{proof}

Suppose \(0 \ne p \in \text{ext}(\partial^-\bar{u}(0))\). Then we can choose \(x_k \in \Omega\) converging to \(0\) such that \(\bar{u}\) is differentiable at \(x_k\) and \(D\bar{u}(x_k) \rightarrow p\) as \(k \rightarrow \infty\). Since \(p \in \partial V\) (by Lemma \ref{lem:pos1}) and \(Dv = \text{id}\) in \(\Omega \setminus U\), we have that \(x_k \in U\). By \eqref{eq:inball}, we have $B_{|x_k-D\bar u(x_k)|}(D\bar u(x_k)) \cap \Omega \subset U.$
Passing to the limit $k\rightarrow \infty,$ we have
\begin{equation}\label{eq:inball2}
  B_{|0-p|}(p) \cap \Omega \subset U.
\end{equation}

Using the interior ball property \eqref{eq:inball}, arguing as in \cite[Lemma 6.8]{CM}, we obtain the following lemma, which shows that the free boundary is never mapped to the free boundary.

\begin{lemma}\label{lem:intco}
  For every \(p \in \text{ext}(\partial^-\bar{u}(0)) \cap \partial V\), we have \(p \in \partial \Omega^*\) and
  \(B_{r_0}(p) \cap \Omega^* \subset V\) for some small \(r_0>0\).
\end{lemma}

\begin{proof}
  Let \(p_0 \in \text{ext}(\partial^-\bar{u}(0)) \cap \partial V\). By \eqref{eq:inball2}, we have
  \(x_t := t\frac{p_0}{|p_0|} \in U\) for \(t > 0\) small.
  Let \(\{x_k\} \subset U\) be a sequence of differentiable points of \(\bar{u}\) such that
  \(x_k \rightarrow x_t\) as \(k \rightarrow \infty\), then up to a subsequence
  \[ D\bar{u}(x_k) \rightarrow p_t \in \partial^-\bar{u}(x_t) \]
  for some \(p_t \in \overline{V}\) as \(k \rightarrow \infty\). By the strict convexity of \(\bar{u}\) in \(\Omega\), we have that \(p_0 \neq p_t\).

  By the interior ball property \eqref{eq:inball}, we have \(B_{|x_k - D\bar{u}(x_k)|}(x_k) \cap \Omega^* \subset V\) for each $k$.
  Letting $k\to\infty$, we obtain
  \begin{equation}\label{eq:mot1}
    B_{|x_t - p_t|}(x_t) \cap \Omega^* \subset V.
  \end{equation}
  By the monotonicity of \(\partial^-\bar{u}\), we have \((p_t - p_0) \cdot (x_t - 0) \geq 0\), which implies that
  \begin{equation}\label{eq:mot2}
    p_0 \in B_{|x_t - p_t|}(x_t).
  \end{equation}
  Hence, from \eqref{eq:mot1} and \eqref{eq:mot2} we obtain \(B_{r_0}(p_0) \cap \Omega^* \subset V\) for some small \(r_0>0\). Combining this with Lemma \ref{lem:pos1}, we have \(p_0 \in \partial \Omega^*\).
\end{proof}

We decompose \(\mathcal{F} = \mathcal{F}_1 \cup \mathcal{F}_2 \cup \mathcal{F}_3\),
where
\begin{align*}
  \mathcal{F}_1 & := \{x \in \mathcal{F} :  \partial^- \bar{u}(x) \text{ contains a vertex of }\Omega^* \},                                        \\
  \mathcal{F}_2 & := \{x \in \mathcal{F}\backslash \mathcal{F}_1 :  \partial^- \bar{u}(x) \text{ touches at least two open edges of } \Omega^* \}, \\
  \mathcal{F}_3 & := \{x \in \mathcal{F} :  \partial^- \bar{u}(x) \text{ touches exactly one open edge of } \Omega^*\}.
\end{align*}

\begin{figure}[htbp]
  \centering
  \includegraphics[width=4in,keepaspectratio]{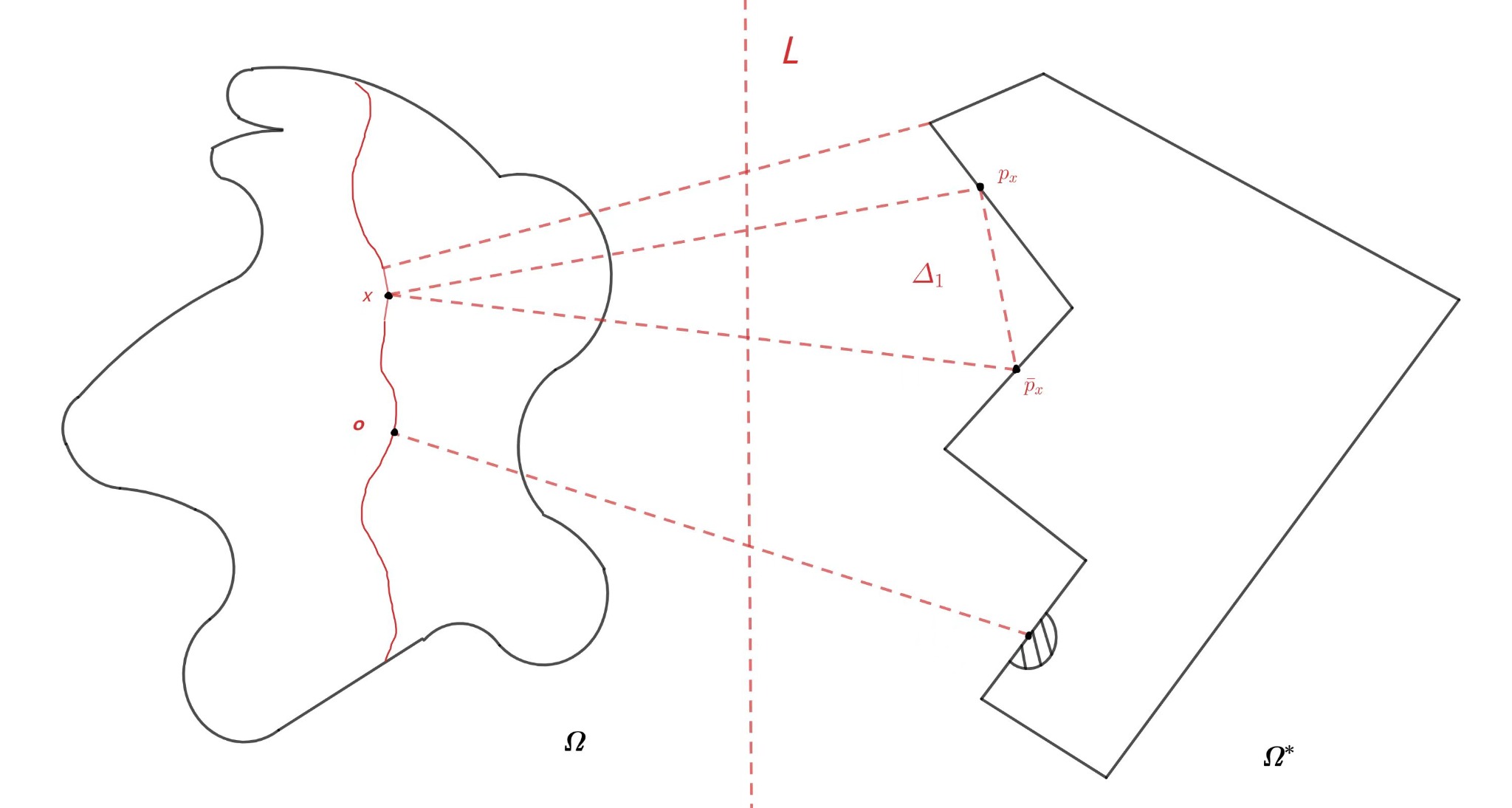}
  \caption{The triangle used in the proof that $\mathcal{F}_2$ is finite.}
  \label{fig:free-boundary-triangle}
\end{figure}

\begin{lemma}
  Both \(\mathcal{F}_1\) and \(\mathcal{F}_2\) consist of finitely many points.
\end{lemma}

\begin{proof}
  For any \( x \in \mathcal{F}_1 \), by the definition of \(\mathcal{F}_1\), there exists a vertex $b_i$ of \(\Omega^*\), \(1 \leq i \leq m\), such that \(Dv(b_i) = x\). Note that by \eqref{eq:c1lemma1}, \( v \) is differentiable at \( b_i \).
  This implies that \(\mathcal{F}_1\) is a finite set with cardinality at most \(m\).

  Let \( x \in \mathcal{F}_2 \). By the definition of \(\mathcal{F}_2\), there are two distinct integers $i, j$, \(1 \leq i \neq j \leq m\), such that \(\partial^- \bar{u}(x) \cap (b_i b_{i+1}) \neq \emptyset\) and \(\partial^- \bar{u}(x) \cap (b_j b_{j+1}) \neq \emptyset\). Denote
  \[ p_x \in \partial^- \bar{u}(x) \cap (b_i b_{i+1}) \quad\text{ and }\quad \bar{p}_x \in \partial^- \bar{u}(x) \cap (b_j b_{j+1}). \]
  Then \(v = \ell\) on the open triangle \(\triangle_1\) with vertices \(x\), \(p_x\), and \(\bar{p}_x\) (see Figure~\ref{fig:free-boundary-triangle}), where \(\ell\) is an affine function with \(D\ell = x\).

  If there exists another \(z \in \mathcal{F}_2\) such that \(\partial^- \bar{u}(z)\) touches the same open edges \((b_i b_{i+1})\) and \((b_j b_{j+1})\) at points \(p_z \in (b_i b_{i+1})\) and \(\bar{p}_z \in (b_j b_{j+1})\), then \(v=\ell'\) for another affine function with \(D\ell' = z\), on the open triangle \(\triangle_2\) with vertices \(z\), \(p_z\), and \(\bar{p}_z\).

  Since \(x, z \in \Omega\), which is separated from \(\Omega^*\) by the straight line \(L\), \(\triangle_1\) and \(\triangle_2\) must intersect, contradicting \eqref{eq:c1lemma1}, which says that \(v\) is differentiable in \(\triangle_1 \cap \triangle_2\). Therefore, different points in \(\mathcal{F}_2\) must correspond to different pairs of open edges of \(\Omega^*\), implying that \(\mathcal{F}_2\) is a finite set with cardinality at most \(\frac{m(m-1)}{2}\).
\end{proof}

It remains to consider the set \(\mathcal{F}_3\). Let \(x_0 \in \mathcal{F}_3\). After translating coordinates, we may assume \(x_0 = 0\).
By the definition of \(\mathcal{F}_3\), we have that \(\partial^- \bar{u}(0) \cap \partial V \subset (b_{i_0} b_{i_0+1})\) for some integer $i_0$, \(1 \leq i_0 \leq m\). Let \(p_0 \in \text{ext}(\partial^- \bar{u}(0)) \cap \partial V\).
Choose \(r_0\) small enough such that
\begin{equation*}
  r_0 < \min\{\text{dist}(p_0, [b_i b_{i+1}]): 1 \leq i \leq m, \ p_0 \notin [b_i b_{i+1}]\},
\end{equation*}
and that \(B_{r_0}(p_0) \cap \Omega^* \subset V\).
Similar to \eqref{eq:loc1}, there exist \(h_0, \delta_0\) such that
\begin{equation}\label{eq:loc333}
  S_h^c[v](y) \subset B_{r_0}(p_0) \quad \forall\, y \in B_{\delta_0}(p_0), \ \forall\, h \leq h_0.
\end{equation}

By the same proof as in Lemma \ref{lem:locloc}, we have the following lemma.
\begin{lemma} \label{lem:locloc2}
  For any $q \in  \partial^- u(0)\cap \partial \Omega^*,$ there exists $0<r$ such that the function \(v\) satisfies
  \begin{equation}\label{eq:alex112}
    \det\, D^2 v = \chi_{_{B_{r}(q) \cap \Omega^*}} \ \text{in}\ B_{r}(q) \ \text{in the Alexandrov sense}.
  \end{equation}
\end{lemma}

\begin{proof}[Proof of Theorem \ref{thm:t222}]
  By \eqref{eq:loc333} and \eqref{eq:alex112}, we can invoke the proof of \cite[Theorem 7.13]{CM} to obtain a quantitative uniform convexity of \( v \) within \( B_{\delta_0}(p_0) \cap \overline{\Omega^*} \). Specifically, this implies that for any two points \( y, \tilde{y} \in B_{\delta_0}(p_0) \cap \overline{\Omega^*} \), the following inequality holds:
  \begin{equation*}
    |Dv(y) - Dv(\tilde{y})| \geq C|y - \tilde{y}|^a
  \end{equation*}
  for some constants \( a > 2 \) and \( C>0 \).
  Since \(\bar{u} = v^*\), we can apply the result of \cite[Remark 7.10]{CM} to deduce that \(\bar{u} \in C^{1,\beta}(B_r(0) \cap \overline{U})\) for some \(\beta \in (0, 1)\) and \(r > 0\).

  By \eqref{eq:inball}, we have that the unit normal of \(\mathcal{F}\) at \(x\) is given by
  \[ \nu_{\mathcal{F}}(x) = \frac{Du(x) - x}{|Du(x) - x|}. \]
  Then, the \(C^{1, \beta}\) regularity of \(\bar{u}\) implies that \(\mathcal{F} \cap B_r(0)\) is \(C^{1,\beta}\), noting that $\bar u=u$ on $U$.

  Since \(0 \in \mathcal{F}_3\), the definition of \(\mathcal{F}_3\), together with the strict convexity of \(v\) in \(B_{\delta_0}(p_0) \cap \overline{\Omega^*}\), implies that \(\partial^- \bar{u}(0) \cap \partial V = \{p_0\}\).
  It follows that
  \[ \nu_{_U}(0) \cdot \nu_{_V}(p_0) > 0, \]
  where \(\nu_{_V}(p_0)\) is the unit inner normal of \(V\) at \(p_0\), and \(\nu_{_U}(0) = \nu_{\mathcal{F}}(0)\) is the unit inner normal of \(U\) at \(0\), namely, the obliqueness estimate holds at \(0\).

  Finally, similar to the proof of Proposition \ref{prop:propmain}, we can apply the arguments of \cite[Sections 3 and 4]{CLW3} to show that
  \(\bar{u} \in C^\infty(B_r(0) \cap \overline{U})\), which implies that \(\mathcal{F} \cap B_r(0)\) is smooth.
  Therefore, \(\mathcal{F}\) is smooth away from finitely many points in \(\mathcal{F}_1 \cup \mathcal{F}_2\).
\end{proof}

\begin{remark}
  We note that in Theorems \ref{thm:t111} and \ref{thm:t222}, \(\Omega^*\) does not need to be simply connected and may contain finitely many ``polygonal'' holes. The proofs of Theorems \ref{thm:t111} and \ref{thm:t222} are also valid for these more general cases. The method developed in this paper is also applicable when the edges of the target domain are replaced by smooth convex curves.
\end{remark}

\section{Higher dimensions}\label{sec:S5}
In this section, we discuss two conjectures in higher dimensions. The first one is for complete transport. Denote \(\Sigma = \overline{A} \cap \Omega.\)

\begin{conjecture}\label{conj:t1111}
  Suppose \(\Omega\) is a bounded domain and \(\Omega^*\) is a non-convex polytope in \(\mathbb{R}^n\) with \(|\Omega|=|\Omega^*|\). Suppose \( u \) is a convex solution to \eqref{eq:bre1}. Then, \(\Sigma\) is either empty or locally an \((n-1)\)-dimensional smooth hypersurface away from a set of Hausdorff dimension at most \(n-2\).
\end{conjecture}

Similar to the two-dimensional case, we may show that \(\partial^- u(x_0)\) is a bounded convex polytope with vertices on \(\partial\Omega^*\), provided \(x_0 \in A.\) We decompose \(\Sigma\) into three parts: \(\Sigma = \Sigma_{\text{reg}} \cup \Sigma_{\text{sin}} \cup \Sigma_{\text{cu}},\) where
\(x \in \Sigma_{\text{reg}}\) if and only if \(\partial^- u(x)\) is a segment with endpoints on two \((n-1)\)-dimensional facets of \(\Omega^*\), \(\Sigma_{\text{sin}} := A \setminus \Sigma_{\text{reg}},\) and \(\Sigma_{\text{cu}} := \Sigma \setminus (\Sigma_{\text{reg}} \cup \Sigma_{\text{sin}}).\) We also write \(\Sigma_{\text{sin}} = \Sigma'_{\text{sin}} \cup \Sigma''_{\text{sin}},\) where \(x \in \Sigma''_{\text{sin}}\) if and only if \(\partial^- u(x)\) is at least two-dimensional, and \(\Sigma'_{\text{sin}} = \Sigma_{\text{sin}} \setminus \Sigma''_{\text{sin}}.\)

The method developed in this paper can be applied to show that \(\Sigma_{\text{reg}}\) is locally a smooth hypersurface, and the Hausdorff dimension of \(\Sigma''_{\text{sin}}\) is at most \(n-2\). Hence, it remains to show that both \(\Sigma'_{\text{sin}}\) and \(\Sigma_{\text{cu}}\) have Hausdorff dimension at most \(n-2\).

Similarly, for optimal partial transport, we have the following conjecture in higher dimensions.

\begin{conjecture}\label{conj:t2222}
  Suppose \(\Omega\) is a bounded domain and \(\Omega^*\) is a non-convex polytope in \(\mathbb{R}^n\), separated from \(\Omega\) by a hyperplane. Let \(m < \min\Big\{|\Omega|, \, |\Omega^*|\Big\}\) be the mass transported. Then, the free boundary is locally smooth away from a set of Hausdorff dimension at most \(n-2\).
\end{conjecture}

\section{Appendix}\label{sec:S6}
Let \(U \subset \mathbb{R}^2\) be a bounded open set, and let \(V \subset \mathbb{R}^2\) be a convex domain, satisfying \(|U| = |V|\).
Suppose \(0 \in \partial U\) and
\begin{equation}\label{eq:betaf}
  \partial U \cap B_{r}(0)\ \text{is } C^{1,\beta}\text{-regular for some }\beta \in (0, 1).
\end{equation}
Suppose \(0 \in \partial V\) and \(\partial V \cap B_{r}(0)\) is a segment. Let \(e_1, e_2\) be the coordinate directions of the standard orthonormal basis of \(\mathbb{R}^2\).
Suppose \(\nu_{_U}(0) = \nu_{_V}(0) = e_2\), where \(\nu_{_U}(0)\) and \(\nu_{_V}(0)\) denote the unit inner normals of \(U\) and \(V\) at 0, respectively.

Let \(u\) be a convex function satisfying
\begin{equation*}
  \left\{
  \begin{array}{rl}
    \det\, D^2 u \!\! & = \ \chi_{_U} \quad \text{in } \mathbb{R}^2, \\
    Du(U) \!\!        & = \ V.
  \end{array}
  \right.
\end{equation*}
This implies that \((Du)_{\sharp}\chi_{_U} = \chi_{_V}\).

Let \(v\) be a convex function satisfying \((Dv)_{\sharp}\chi_{_V} = \chi_{_U}\). Note that \(v = u^*\) in \(V\).
Assume further that
\begin{equation}\label{eq:veq1}
  \det\, D^2 v = \chi_{_{B_r(0) \cap V}} \quad \text{in}\ B_r(0),
\end{equation}
and that \(Du(0) = Dv(0) = 0\), \(v(0) = u(0) = 0\). Suppose that
\[ S_h^c[v](y) \subset B_r(0) \quad \forall\, y \in B_{\delta_0}(0)\cap V,\ \ \forall\, h \leq h_0, \]
where \(h_0\) and \(\delta_0\) are small positive constants. The arguments developed in \cite[Sections 3 and 4]{CLW3} lead to the following useful proposition. For the reader's convenience, we will sketch the proof here.

\begin{proposition}\label{prop:aprop1}
  Suppose \(u, v, U, V\) satisfy all the conditions listed above. Then,\\
  (i) \(u(x) \leq C_\epsilon |x|^{2-\epsilon}\) for \(x \in \overline{U}\), \(\forall\,\epsilon>0\) \\
  (ii) \(u \in C^{2, \alpha}(B_{\bar{r}}(0) \cap \overline{U})\) for some \(\alpha \in (0, 1)\) and \(\bar{r} > 0\) small,
  provided \(\beta\) (appearing in \eqref{eq:betaf}) is sufficiently close to 1.
\end{proposition}

\begin{proof}
  Since \(\det\, D^2 v = \chi_{_{S_{h_0}^c[v](0) \cap V}}\) and \(S_{h_0}^c[v](0) \cap \partial V\) is a segment, it follows from
  \cite[Corollary 1.1]{C96} that  \(\|v_{ii}\|_{L^\infty(B_{r_1}(0))} \leq C\) for some small \(r_1 > 0\) and a constant \(C > 0\).
  It follows that \(|v(x_1, 0)| \leq C|x_1|^2\) for some constant \(C > 0\).
  Then, by the same proof as in \cite[Lemma 3.1]{CLW3}, we have that
  \begin{equation}\label{eq:ae1}
    u(x) \geq C_\epsilon |x_1|^2 \quad \text{for} \ x = (x_1, x_2) \in U \ \text{near} \ 0.
  \end{equation}
  Following the same proof as in \cite[Lemma 3.2]{CLW3}, we have that
  \begin{equation}\label{eq:ae2}
    u(0, t) \leq C_\epsilon |t|^{2-\epsilon} \quad \text{for} \ |t| \ \text{small}.
  \end{equation}
  By \eqref{eq:ae1} and \eqref{eq:ae2}, we can apply the proof of \cite[Lemma 3.3]{CLW3} to obtain the uniform density estimate, namely,
  \begin{equation*}
    \frac{|S_h^c[u](0)  \cap U|}{|S_h^c[u](0) |} \geq \delta
  \end{equation*}
  for \(h \leq h_0\), where \(\delta>0\) is a constant independent of \(h\).
  Once these estimates are available, we can follow the proof of \cite[Lemma 3.4]{CLW3} to derive the tangential \(C^{1, 1-\epsilon}\) estimate for \(u\), namely,
  \begin{equation*}
    B_{C_\epsilon h^{\frac{1}{2}+\epsilon}}(0) \cap \{x_2 = 0\} \subset S_h^c[u](0) .
  \end{equation*}
  Then, by the proof of \cite[Corollary 3.2]{CLW3}, we obtain \(u(x) \leq C_\epsilon |x|^{2-\epsilon}\).
  Part $(ii)$ of the proposition follows by a perturbation argument similar to that in \cite[Section 4]{CLW3}.
\end{proof}

\bibliographystyle{amsplain}

\end{document}